\documentclass[reqno]{amsart}
\usepackage{amscd,amsfonts,amssymb}
\numberwithin{equation}{section}
\newtheorem{Theorem}{Theorem}[section]
\newtheorem{Lemma}{Lemma}[section]
\newtheorem{Corollary}{Corollary}[section]

\theoremstyle{definition}
\newtheorem{Definition}{Definition}[section]
\theoremstyle{remark}
\newtheorem{Remark}{Remark}[section]

\author{A.A. Kon'kov}
\address{Department of Differential Equations,
Faculty of Mechanics and Mathematics,
Mo\-s\-cow Lo\-mo\-no\-sov State University,
Vorobyovy Gory,
Moscow, 119992 Russia}
\email{konkov@mech.math.msu.su}
\author{A.E. Shishkov}
\address{
Center of Nonlinear Problems of Mathematical Physics,
RUDN University,
Miklukho-Maklaya str. 6,
Moscow, 117198 Russia;
Institute of Applied Mathematics and Mechanics of NAS of Ukraine,
Dobrovol'skogo str. 1, Slavyansk, 84116 Ukraine
}
\email{aeshkv@yahoo.com}

\title[On large time behavior]{On large time behavior of solutions of higher order evolution inequalities with fast diffusion}

\keywords{Higher order differential inequalities; Nonlinearity; Large time estimates; Stabilization of solutions}
\subjclass{35B40, 35G20, 35K25, 35K55, 35K65} 
\date{}

\begin{document}

\begin{abstract}
We obtain stabilization conditions and large time estimates for weak solutions of the inequality 
$$
	\sum_{|\alpha| = m}
	\partial^\alpha
	a_\alpha (x, t, u)
	-
	u_t
	\ge
	f (x, t) g (u)
	\quad
	\mbox{in } \Omega \times (0, \infty),
$$
where $\Omega$ is a non-empty open subset of ${\mathbb R}^n$, $m, n \ge 1$, and $a_\alpha$ are Caratheodory functions such that
$$
	|a_\alpha (x, t, \zeta)| 
	\le 
	A \zeta^p,
	\quad
	|\alpha| = m,
$$
with some constants $A > 0$ and $0 < p < 1$ for almost all $(x, t) \in \Omega \times (0, \infty)$ and for all $\zeta \in [0, \infty)$.
For solutions of homogeneous differential inequalities, we give an exact universal upper bound.
\end{abstract}

\maketitle

\section{Introduction}
We study non-negative weak solutions of the inequality
\begin{equation}
	\sum_{|\alpha| = m}
	\partial^\alpha
	a_\alpha (x, t, u)
	-
	u_t
	\ge
	f (x, t) g (u)
	\quad
	\mbox{in } \Omega \times (0, \infty),
	\quad
	\label{1.1}
\end{equation}
where $\Omega$ is a non-empty open subset of ${\mathbb R}^n$, $m, n \ge 1$, and $a_\alpha$ are Caratheodory functions such that
\begin{equation}
	|a_\alpha (x, t, \zeta)| 
	\le 
	A \zeta^p,
	\quad
	|\alpha| = m,
	\label{1.2}
\end{equation}
with some constants $A > 0$ and $0 < p < 1$ 
for almost all 
$(x, t) \in {\mathbb R}_+^{n+1}$
and for all $\zeta \in [0, \infty)$.

As is customary, by ${\mathbb R}_+^{n+1} = {\mathbb R}^n \times (0, \infty)$ 
we denote the upper half-space in ${\mathbb R}^{n+1}$.
In so doing, by $\alpha = {(\alpha_1, \ldots, \alpha_n)}$ 
we mean a multi-index with
$|\alpha| = \alpha_1 + \ldots + \alpha_n$
and
$\partial^\alpha = {\partial^{|\alpha|} / (\partial_{x_1}^{\alpha_1} \ldots \partial_{x_n}^{\alpha_n}})$, where
$x = {(x_1, \ldots, x_n)} \in {\mathbb R}^n$.
Let us also denote by $B_r^x$ the open ball in ${\mathbb R}^n$ of radius $r > 0$ centered at $x$.
In the case of $x = 0$, we write $B_r$ instead of $B_r^0$.
Throughout the paper, it is assumed that $f$ is a non-negative measurable function and
$g \in C ([0, \infty))$ is a non-negative function such that
$g (\zeta) > 0$ for all $\zeta \in (0, \infty)$.

\begin{Definition}\label{D1.1}
A non-negative function $u \in L_{1, loc} (\Omega \times (0, \infty))$ 
is called a weak solution of~\eqref{1.1} if 
$f (x, t) g (u) \in L_{1, loc} (\Omega \times (0, \infty))$ and
\begin{equation}
	\int_0^\infty
	\int_\Omega
	\sum_{|\alpha| = m}
	(- 1)^m
	a_\alpha (x, t, u)
	\partial^\alpha
	\varphi
	\,
	dx
	dt
	+
	\int_0^\infty
	\int_\Omega
	u
	\varphi_t
	\,
	dx
	dt
	\ge
	\int_0^\infty
	\int_\Omega
	f (x, t) g (u)
	\varphi
	\,
	dx
	dt
	\label{1.3}
\end{equation}
for any non-negative function $\varphi \in C_0^\infty (\Omega \times (0, \infty))$.
\end{Definition}

\begin{Definition}\label{D1.2}
A non-negative function $u \in L_{1, loc} (\Omega \times (0, \infty))$ 
is called a weak solution of the equation
\begin{equation}
	\sum_{|\alpha| = m}
	\partial^\alpha
	a_\alpha (x, t, u)
	-
	u_t
	=
	f (x, t) g (u)
	\quad
	\mbox{in } \Omega \times (0, \infty)
	\quad
	\label{1.4}
\end{equation}
if $f (x, t) g (u) \in L_{1, loc} (\Omega \times (0, \infty))$ and
$$
	\int_0^\infty
	\int_\Omega
	\sum_{|\alpha| = m}
	(- 1)^m
	a_\alpha (x, t, u)
	\partial^\alpha
	\varphi
	\,
	dx
	dt
	+
	\int_0^\infty
	\int_\Omega
	u
	\varphi_t
	\,
	dx
	dt
	=
	\int_0^\infty
	\int_\Omega
	f (x, t) g (u)
	\varphi
	\,
	dx
	dt
$$
for any function $\varphi \in C_0^\infty (\Omega \times (0, \infty))$.
\end{Definition}

It can easily be seen that every solution of~\eqref{1.4} is also a solution of~\eqref{1.1}.
A partial case of~\eqref{1.4} is the nonlinear diffusion equation
$$
	\Delta u^p - u_t = f (x, t) g (u).
$$
This equation can be written in the form
\begin{equation}
	u_t 
	- 
	\operatorname{div}
	\left(
		D
		\nabla u
	\right)
	=
	q (x, t),
	\label{1.5}
\end{equation}
where $D = p u^{p - 1}$ is a diffusion coefficient which depends on the density $u$ 
in a power-law manner and $q (x, t) = - f (x, t) g (u)$ is a source-density function.
If $u\ll 1$, then we obviously have $D \ll 1$ for $p > 1$ and $D \gg 1$ for $p < 1$.
In the first case, it is customary to say that~\eqref{1.5} is the slow diffusion equation. 
In the second case,~\eqref{1.5} is the fast diffusion equation.

The  questions treated in the present paper were previously considered mainly for second-order differential operators~[1--18]. 
The case of higher order operators was studied in~\cite{B, KS}.
In so doing, paper~\cite{B} deals with solutions of higher order equations of the $p$-Laplacian type satisfying some initial condition and zero Dirichlet boundary conditions on $\partial \Omega \times (0, \infty)$. 
In~\cite{KS}, the case of slow diffusion was studied.
In our paper, we investigate the fast diffusion case.
We obtain stabilization conditions and estimates at infinity for weak solutions of~\eqref{1.1} in $L_1$-norm. In so doing, no ellipticity conditions on the coefficients $a_\alpha$ of the differential operator are imposed. 
We also impose no initial or boundary conditions on solutions of~\eqref{1.1}.
Thus, our results can be applied to a wide class of differential inequalities. 
For solutions of homogeneous differential inequalities we obtain sharp universal upper bound 
(see Corollary~\ref{C2.1} and Remark~\ref{R2.2}).

\section{Main results}

\begin{Theorem}\label{T2.1}
Suppose that
\begin{equation}
	\liminf_{\zeta  \to \infty} 
	\frac{
		g (\zeta)
	}{
		\zeta
	}
	>
	0
	\label{T2.1.1}
\end{equation}
and 
\begin{equation}
	\lim_{t \to \infty}
	\mathop{\rm ess\,inf}_{
		\Omega \times (t, \infty)
	}
	f 
	=
	\infty.
	\label{T2.1.2}
\end{equation}
Then any non-negative weak solution of~\eqref{1.1} stabilizes to zero as $t \to \infty$ 
in the $L_1$ norm on an arbitrary compact set $K \subset \Omega$, i.e.
\begin{equation}
	\limsup_{t \to \infty}
	\| u (\cdot, t) \|_{
		L_1 (K)
	}
	=
	0.
	\label{T2.1.3}
\end{equation}
\end{Theorem}

\begin{Theorem}\label{T2.2}
Suppose that~\eqref{T2.1.1} is valid and
$$
	\lim_{t \to \infty}
	\mathop{\rm ess\,inf}_{
		\Omega \times (t, \infty)
	}
	f 
	\ge
	\gamma,
$$
where $\gamma > 0$ is a real number.
Then
\begin{equation}
	\limsup_{t \to \infty}
	\| u (\cdot, t) \|_{
		L_1 (K)
	}
	\le
	C,
	\label{T2.2.3}
\end{equation}
for any non-negative weak solution of~\eqref{1.1} and compact set $K \subset \Omega$,
where the constant $C > 0$ depends only on $A$, $\gamma$, $m$, $n$, $p$, $K$, $\Omega$, and the function $g$.
\end{Theorem}

The proof of Theorems~\ref{T2.1} and~\ref{T2.2} is given in Section~\ref{proof}.

\begin{Remark}\label{R2.1}
Since $u \in L_{1, loc} ({\mathbb R}_+^{n+1})$, the norm on the left in~\eqref{T2.1.3} and~\eqref{T2.2.3}
is defined for almost all $t \in (0, \infty)$;
therefore, the limit in~\eqref{T2.1.3} and~\eqref{T2.2.3} is understood in the essential sense.
In other words, 
$$
	\limsup_{t \to \infty}
	\| u (\cdot, t) \|_{
		L_1 (K)
	}
	=
	\inf \Lambda,
$$
where $\Lambda$ is the set of real numbers $\lambda$ such that
$$
	\| u (\cdot, t) \|_{
		L_1 (K)
	}
	<
	\lambda
$$
for almost all $t$ in a neighborhood of infinity. 
In the case of $\Lambda = \emptyset$, we write 
$$
	\limsup_{t \to \infty}
	\| u (\cdot, t) \|_{
		L_1 (K)
	}
	=
	\infty.
$$
\end{Remark}

\begin{Corollary}\label{C2.1}
Suppose that $u$ is a non-negative weak solution of the inequality
\begin{equation}
	\sum_{|\alpha| = m}
	\partial^\alpha
	a_\alpha (x, t, u)
	-
	u_t
	\ge
	0
	\quad
	\mbox{in } B_{2 R}^y \times (0, \infty),
	\label{C2.1.1}
\end{equation}
where $y \in {\mathbb R}^n$ and $R > 0$ is a real number.
Then
\begin{equation}
	\| u (\cdot, t) \|_{
		L_1 (B_R^y)
	}
	\le
	C
	R^{
		n - m / (1 - p)
	}
	t^{1 / (1 - p)}
	\label{C2.1.2}
\end{equation}
for almost all $t$ in a neighborhood of infinity, 
where the constant $C > 0$ depends only on $A$, $m$, $n$, and $p$.
\end{Corollary}

\begin{proof}
Putting
$$
	\tilde u (x, t) 
	= 
	e^{- t / (1 - p)} 
	R^{
		m / (1 - p)
	}
	u (R x + y, e^t),
$$
we obtain a weak solution of the inequality
$$
	\sum_{|\alpha| = m}
	\partial^\alpha
	\tilde a_\alpha (x, t, \tilde u)
	-
	\tilde u_t
	\ge
	\frac{1}{1 - p}
	\tilde u
	\quad
	\mbox{in } B_2 \times (0, \infty),
$$
where
$$
	\tilde a_\alpha (x, t, \zeta)
	=
	e^{- p t / (1 - p)}
	R^{p m / (1 - p)}
	a_\alpha 
	(
		R x + y, 
		e^t, 
		e^{t / (1 - p)} 
		R^{
			- m / (1 - p)
		}
		\zeta
	),
	\quad
	|\alpha| = m.
$$
From~\eqref{1.2}, it follows that
$$
	|\tilde a_\alpha (x, t, \zeta)| 
	\le 
	A \zeta^p,
	\quad
	|\alpha| = m,
$$
for almost all $(x, t) \in {\mathbb R}_+^{n+1}$ and for all $\zeta \in [0, \infty)$.
Thus, by Theorem~\ref{T2.2}, we have
$$
	\| \tilde u (\cdot, t) \|_{
		L_1 (B_1)
	}
	\le
	C
$$
for almost all $t$ in a neighborhood of infinity, 
where the constant $C > 0$ depends only on $A$, $m$, $n$, and $p$.
This in turn implies~\eqref{C2.1.2}.
\end{proof}

\begin{Remark}\label{R2.2}
Inequality~\eqref{C2.1.2} is the best possible universal upper estimate for weak solutions
of~\eqref{C2.1.1}. Indeed, if
$$
	0 
	< 
	p 
	< 
	p^*
	=
	1 - \frac{m}{n},
$$
where $m$ is an even integer, then
there exists a constant $u_0 > 0$ depending only on $m$, $n$, and $p$ such that 
$$
	u (x, t)
	=
	u_0
	|x|^{- m / (1 - p)}
	t^{1 / (1 - p)}
$$
is a weak solution of the backward fast diffusion equation
$$
	(- \Delta)^{m / 2} u^p 
	- 
	u_t 
	=
	0
	\quad
	\mbox{in } {\mathbb R}_+^{n+1}.
$$
In so doing, it is easy to verify that
$$
	\| u (\cdot, t) \|_{
		L_1 (B_R)
	}
	=
	C
	R^{
		n - m / (1 - p)
	}
	t^{1 / (1 - p)}
$$
for all $R > 0$ and $t > 0$, 
where the constant $C > 0$ depends only on $m$, $n$, and $p$.

For $m = 2$ and $n > 2$, the critical exponent
$$
	p_* = 1 - \frac{2}{n}
$$
plays an important role in the study of solutions of the fast diffusion equation
$$
	\Delta u^p - u_t = 0
	\quad
	\mbox{in } {\mathbb R}_+^{n+1}.
$$
The fundamental solution of this equation with an arbitrary initial mass
$$
	u (x, 0) = \kappa \delta (x),
	\quad
	\kappa > 0,
$$
where $\delta (x)$ is the Dirac measure, exists if and only if $p^* < p < 1$.
This solution is known as the Barenblatt--Zel'dovich--Kompaneets solution
$$
	u_\kappa (x, t)
	=
	t^{-l}
	\left(
		C_\kappa 
		+
		\frac{
			(1 - p) l
		}{
			2 p n
		}
		\frac{
			|x|^2
		}{
			t^{2 l / n}
		}
	\right)^{-1 / (1 - p)},
$$
where
$$
	l = \frac{n}{2 - n (1 - p)} > 0
	\quad
	\mbox{and}
	\quad
	C_\kappa
	=
	a (n, p)
	\kappa^{- 2 (1 - p) l / n},
	\quad
	a (n, p) > 0.
$$
It can be seen that $C_\kappa \to 0$ as $\kappa \to \infty$ and the limit function
$$
	u_\infty (x, t)
	:=
	\lim_{\kappa \to \infty}
	u_\kappa (x, t)
	=
	C_*
	\left(
		\frac{t}{|x|^2}
	\right)^{1 / (1 - p)},
	\quad
	C_*
	=
	\left(
		\frac{
			2 p (2 - n + n p)
		}{
			1 - p
		}
	\right)^{1 / (1 - p)},
$$
satisfies the equation
$$
	\Delta u_\infty^p - \partial_t u_\infty = 0
	\quad
	\mbox{in } {\mathbb R}_+^{n+1} \setminus S,
$$
where
$
	S
	=
	\{
		(x, t) \in {\mathbb R}_+^{n+1}
		:
		x = 0
	\}.
$
In particular, 
$$
	\Delta u_\infty^p - \partial_t u_\infty = 0
	\quad
	\mbox{in } B_{2 R}^y \times (0, \infty),
$$
for any $y \in {\mathbb R}^n \setminus \{ 0 \}$, where $R = |y| / 2$.
In so doing, by direct calculations, it can be shown that
$$
	\| u_\infty (\cdot, t) \|_{
		L_1 (B_R^y)
	}
	=
	C
	R^{
		n - 2 / (1 - p)
	}
	t^{1 / (1 - p)}
$$
for all $t > 0$, where the constant $C > 0$ depends only on $n$ and $p$.
Thus, estimate~\eqref{C2.1.2} is the best possible.

The function $u_\infty$ has a strong singularity at the axis $S$
and is called a razor blade solution~\cite{CV, VV}.
Razor blade solutions can appear for equations of different structure, for example, as solutions of elliptic and parabolic equations with slowly and fast diffusion and of equations with nonlinear degenerate absorption~\cite{MS, SVRLMA}.

We note that Corollary~\ref{C2.1} guarantees the validity of estimate~\eqref{C2.1.2} 
only for $t$ from a neighborhood of infinity which, generally speaking, depends on $R$ and~$u$.
This limitation can not be relaxed. 
Indeed, in the case of $p_* < p < 1$, 
the Barenblatt--Zel'dovich--Kompaneets solutions $u_\kappa$ satisfy the equation
$$
	\Delta u_\kappa^p - \partial_t u_\kappa = 0
	\quad
	\mbox{in }
	{\mathbb R}_+^{n+1}.
$$
Hence, by Corollary~\ref{C2.1}, for any $\kappa > 0$ and $R > 0$ there exists $t_* > 0$ 
such that
$$
	\| u_\kappa (\cdot, t) \|_{
		L_1 (B_R)
	}
	\le
	C
	R^{
		n - 2 / (1 - p)
	}
	t^{1 / (1 - p)}
$$
for almost all $t > t_*$, where the constant $C > 0$ depends only on $n$ and $p$.
Since
$$
	\lim_{R \to \infty}
	\| u_\kappa (\cdot, t) \|_{
		L_1 (B_R)
	}
	=
	\kappa
$$
for all $\kappa > 0$ and $t > 0$ and, moreover,
$$
	\lim_{R \to \infty}
	R^{
		n - 2 / (1 - p)
	}
	=
	0,
$$
the real number $t_*$ can not be the same for all $R > 0$.
\end{Remark}

\section{Proof of Theorems~\ref{T2.1} and~\ref{T2.2}}\label{proof}

In this section, it is assumed that $u$ is a non-negative solution of~\eqref{1.1}.
We take the Steklov-Schwartz averaging kernel
$$
	\omega_h (t) 
	= 
	\frac{1}{h} 
	\omega
	\left(
		\frac{t}{h}
	\right),
	\quad
	h > 0,
$$
where $\omega \in C^\infty ({\mathbb R})$ is a non-negative function such that
$\operatorname{supp} \omega \subset (-1, 1)$ and
$$
	\int_{-\infty}^\infty
	\omega
	\,
	d t
	=
	1.
$$
For any measurable set $M$ with compact closure belonging to $\Omega$ 
we denote by $T (M)$ the set of $\tau \in (0, \infty)$
such that $u (\cdot, \tau) \in L_1 (M)$ and
$$
	\int_{M \times (0, \infty)}
	\omega_h (\tau - t)
	u (x, t)
	\,
	dx dt
	=
	\int_0^\infty
	\omega_h (\tau - t)
	\int_M
	u (x, t)
	\,
	dx dt
	\to
	\int_M
	u (x, \tau)
	\,
	dx
	\quad
	\mbox{as } h \to +0.
$$
Since $u \in L_{1, loc} (\Omega \times (0, \infty))$, the Lebesgue measure of 
the difference $(0, \infty) \setminus T (M)$ is equal to zero.

\begin{Lemma}\label{L3.1}
Let $0 < R \le R_1 < R_2 \le 2 R$ and $t_1 < t_2$ be real numbers such that
$\overline{B_{2 R}^y} \subset \Omega$,
$
	t_1 
	\in 
	T 
	(
		B_{2 R}^y
	),
$
and
$
	t_2 
	\in 
	T 
	(
		B_R^y
	)
$
for some $y \in \Omega$. 
Then
\begin{align*}
	&
	\int_{
		B_{2 R}^y
	}
	e^{\lambda t_1}
	u (x, t_1)
	\,
	dx
	-
	\int_{
		B_R^y
	}
	e^{\lambda t_2}
	u (x, t_2)
	\,
	dx
	+
	\lambda
	\int_{
		B_{R_2}^y \times (t_1, t_2)
		\setminus
		E
	}
	e^{\lambda t}
	u
	\,
	dx dt
	\\
	&
	\qquad
	{}
	+
	\frac{
		C
	}{
		(R_2 - R_1)^m
	}
	\int_{
		(B_{R_2}^y	\setminus B_{R_1}^y) \times (t_1, t_2)
	}
	e^{\lambda t}
	u^p
	\,
	dx	dt
	\ge
	\frac{1}{2}
	\int_{
		B_{R_1}^y \times (t_1, t_2)
		\cap
		E
	}
	e^{\lambda t}
	f (x, t) g (u)
	\,
	dx dt
\end{align*}
for any measurable set $E \subset B_{2 R}^y \times (t_1, t_2)$ and real number $\lambda > 0$ 
such that
\begin{equation}
	f (x, t) g (u)
	\ge
	2
	\lambda u 
	\quad
	\mbox{on } E,
	\label{L3.1.2}
\end{equation}
where the constant $C > 0$ depends only on $A$, $m$, $n$, and $p$.
\end{Lemma}

\begin{proof}
Without loss of generality, it can be assumed that $y = 0$.
Consider a non-decreasing function $\varphi_0 \in C^\infty ({\mathbb R})$ 
satisfying the conditions
$$
	\left.
		\varphi_0
	\right|_{
		(- \infty, 0]
	}
	=
	0
	\quad
	\mbox{and}
	\quad
	\left.
		\varphi_0
	\right|_{
		[1, \infty)
	}
	=
	1.
$$
Taking
$$
	\varphi (x, t)
	=
	\varphi_0
	\left(
		\frac{R_2 - |x|}{R_2 - R_1}
	\right)
	\omega_h (\tau - t)
	e^{\lambda t}
$$
as a test function in~\eqref{1.3}, we obtain
\begin{align*}
	&
	\int_0^\infty
	\int_\Omega
	\sum_{|\alpha| = m}
	(- 1)^m
	a_\alpha (x, t, u)
	\partial^\alpha
	\varphi_0
	\left(
		\frac{R_2 - |x|}{R_2 - R_1}
	\right)
	\omega_h (\tau - t)
	e^{\lambda t}
	\,
	dx
	dt
	\\
	&
	\quad
	{}
	-
	\int_0^\infty
	\int_\Omega
	u
	\varphi_0
	\left(
		\frac{R_2 - |x|}{R_2 - R_1}
	\right)
	\omega_h' (\tau - t)
	e^{\lambda t}
	\,
	dx
	dt
	+
	\lambda
	\int_0^\infty
	\int_\Omega
	u
	\varphi_0
	\left(
		\frac{R_2 - |x|}{R_2 - R_1}
	\right)
	\omega_h (\tau - t)
	e^{\lambda t}
	\,
	dx
	dt
	\\
	&
	\qquad
	{}
	\ge
	\int_0^\infty
	\int_\Omega
	f (x, t) g (u)
	\varphi_0
	\left(
		\frac{R_2 - |x|}{R_2 - R_1}
	\right)
	\omega_h (\tau - t)
	e^{\lambda t}
	\,
	dx
	dt
\end{align*}
for all $\tau \in (t_1, t_2)$ and $h \in (0, t_1)$. 
By~\eqref{1.2}, this implies that
\begin{align*}
	&
	\frac{
		C
	}{
		(R_2 - R_1)^m
	}
	\int_0^\infty
	\int_{
		B_{R_2} \setminus B_{R_1}
	}
	u^p 
	\omega_h (\tau - t)
	e^{\lambda t}
	\,
	dx	dt
	\\
	&
	\quad
	{}
	-
	\int_0^\infty
	\int_\Omega
	u
	\varphi_0
	\left(
		\frac{R_2 - |x|}{R_2 - R_1}
	\right)
	\omega_h' (\tau - t)
	e^{\lambda t}
	\,
	dx	dt
	+
	\lambda
	\int_0^\infty
	\int_\Omega
	u
	\varphi_0
	\left(
		\frac{R_2 - |x|}{R_2 - R_1}
	\right)
	\omega_h (\tau - t)
	e^{\lambda t}
	\,
	dx	dt
	\\
	&
	\qquad
	{}
	\ge
	\int_0^\infty
	\int_\Omega
	f (x, t) g (u)
	\varphi_0
	\left(
		\frac{R_2 - |x|}{R_2 - R_1}
	\right)
	\omega_h (\tau - t)
	e^{\lambda t}
	\,
	dx	dt
\end{align*}
for all $\tau \in (t_1, t_2)$ and $h \in (0, t_1)$,
where the constant $C > 0$ depends only on $A$, $m$, $n$, and $p$.
Integrating the last expression with respect to $\tau$ from $t_1$ to $t_2$, we have
\begin{align}
	&
	\frac{
		C
	}{
		(R_2 - R_1)^m
	}
	\int_0^\infty
	\int_{
		B_{R_2}	\setminus B_{R_1}
	}
	u^p 
	\sigma_h (t)
	e^{\lambda t}
	\,
	dx	dt
	\nonumber
	\\
	&
	\quad
	{}
	-
	\int_0^\infty
	\int_\Omega
	u
	\varphi_0
	\left(
		\frac{R_2 - |x|}{R_2 - R_1}
	\right)
	s_h (t)
	e^{\lambda t}
	\,
	dx	dt
	+
	\lambda
	\int_0^\infty
	\int_\Omega
	u
	\varphi_0
	\left(
		\frac{R_2 - |x|}{R_2 - R_1}
	\right)
	\sigma_h (t)
	e^{\lambda t}
	\,
	dx	dt
	\nonumber
	\\
	&
	\qquad
	{}
	\ge
	\int_0^\infty
	\int_\Omega
	f (x, t) g (u)
	\varphi_0
	\left(
		\frac{R_2 - |x|}{R_2 - R_1}
	\right)
	\sigma_h (t)
	e^{\lambda t}
	\,
	dx dt
	\label{PL3.1.1}
\end{align}
for all $h \in (0, t_1)$, where
$$
	\sigma_h (t)
	=
	\int_{t_1}^{t_2}
	\omega_h (\tau - t)
	\,
	d\tau
$$
and
$$
	s_h (t)
	=
	\int_{t_1}^{t_2}
	\omega_h' (\tau - t)
	\,
	d\tau
	=
	\omega_h (t_2 - t)
	-
	\omega_h (t_1 - t).
$$
Since
\begin{align*}
	&
	\int_0^\infty
	\int_\Omega
	u
	\varphi_0
	\left(
		\frac{R_2 - |x|}{R_2 - R_1}
	\right)
	s_h (t)
	e^{\lambda t}
	\,
	dx	dt
	=
	\int_0^\infty
	\omega_h (t_2 - t)
	\int_{
		B_{R_2}
	}
	u
	\varphi_0
	\left(
		\frac{R_2 - |x|}{R_2 - R_1}
	\right)
	e^{\lambda t}
	\,
	dx	dt
	\\
	&
	\quad
	{}
	-
	\int_0^\infty
	\omega_h (t_1 - t)
	\int_{
		B_{R_2}
	}
	u
	\varphi_0
	\left(
		\frac{R_2 - |x|}{R_2 - R_1}
	\right)
	e^{\lambda t}
	\,
	dx	dt
	\\
	&
	\qquad
	{}
	\ge
	\int_0^\infty
	\omega_h (t_2 - t)
	\int_{
		B_R
	}
	u
	e^{\lambda t}
	\,
	dx	dt
	-
	\int_0^\infty
	\omega_h (t_1 - t)
	\int_{
		B_{2 R}
	}
	u
	e^{\lambda t}
	\,
	dx	dt,
\end{align*}
inequality~\eqref{PL3.1.1} yields
\begin{align}
	&
	\frac{
		C
	}{
		(R_2 - R_1)^m
	}
	\int_0^\infty
	\int_{
		B_{R_2} \setminus B_{R_1}
	}
	u^p 
	\sigma_h (t)
	e^{\lambda t}
	\,
	dx	dt
	-
	\int_0^\infty
	\omega_h (t_2 - t)
	\int_{
		B_R
	}
	u
	e^{\lambda t}
	\,
	dx	dt
	\nonumber
	\\
	&
	\quad
	{}
	+
	\int_0^\infty
	\omega_h (t_1 - t)
	\int_{
		B_{2 R}
	}
	u
	e^{\lambda t}
	\,
	dx	dt
	+
	\lambda
	\int_0^\infty
	\int_\Omega
	u
	\varphi_0
	\left(
		\frac{R_2 - |x|}{R_2 - R_1}
	\right)
	\sigma_h (t)
	e^{\lambda t}
	\,
	dx	dt
	\nonumber
	\\
	&
	\qquad
	{}
	\ge
	\int_0^\infty
	\int_\Omega
	f (x, t) g (u)
	\varphi_0
	\left(
		\frac{R_2 - |x|}{R_2 - R_1}
	\right)
	\sigma_h (t)
	e^{\lambda t}
	\,
	dx dt
	\label{PL3.1.2}
\end{align}
for all $h \in (0, t_1)$.
It is easy to see that $0 \le \sigma_h (t) \le 1$ for all $h > 0$ and $t \in (0, \infty)$ and, moreover,
$$
	\sigma_h (t) 
	\to 
	\chi_{
		(t_1, t_2)
	}
	(t)
	\quad
	\mbox{as } h \to +0
$$
for almost all $t \in (0, \infty)$, where
$
	\chi_{
		(t_1, t_2)
	}
$
is the characteristic function of the interval $(t_1, t_2)$.
In so doing,
$$
	\int_0^\infty
	\omega_h (t_2 - t)
	\int_{
		B_R
	}
	u
	e^{\lambda t}
	\,
	dx	dt
	\to
	\int_{
		B_R
	}
	u (x, t_2)
	e^{\lambda t_2}
	\,
	dx
	\quad
	\mbox{as } h \to +0
$$
and
$$
	\int_0^\infty
	\omega_h (t_1 - t)
	\int_{
		B_{2 R}
	}
	u
	e^{\lambda t}
	\,
	dx	dt
	\to
	\int_{
		B_{2 R}
	}
	u (x, t_1)
	e^{\lambda t_1}
	\,
	dx
	\quad
	\mbox{as } h \to +0.
$$
Thus, passing in~\eqref{PL3.1.2} to the limit as $h \to +0$, 
in accordance with Lebesgue's dominated convergence theorem we obtain
\begin{align*}
	&
	\frac{
		C
	}{
		(R_2 - R_1)^m
	}
	\int_{t_1}^{t_2}
	\int_{
		B_{R_2}	\setminus B_{R_1}
	}
	u^p 
	e^{\lambda t}
	\,
	dx	dt
	-
	\int_{
		B_R
	}
	u (x, t_2)
	e^{\lambda t_2}
	\,
	dx	dt
	+
	\int_{
		B_{2 R}
	}
	u (x, t_1)
	e^{\lambda t_1}
	\,
	dx	dt
	\\
	&
	\quad
	{}
	+
	\lambda
	\int_{t_1}^{t_2}
	\int_{
		B_{R_2}
	}
	u
	\varphi_0
	\left(
		\frac{R_2 - |x|}{R_2 - R_1}
	\right)
	e^{\lambda t}
	\,
	dx	dt
	\ge
	\int_{t_1}^{t_2}
	\int_{
		B_{R_2}
	}
	f (x, t) g (u)
	\varphi_0
	\left(
		\frac{R_2 - |x|}{R_2 - R_1}
	\right)
	e^{\lambda t}
	\,
	dx dt,
\end{align*}
whence it follows that
\begin{align*}
	&
	\frac{
		C
	}{
		(R_2 - R_1)^m
	}
	\int_{t_1}^{t_2}
	\int_{
		B_{R_2} \setminus B_{R_1}
	}
	u^p 
	e^{\lambda t}
	\,
	dx	dt
	-
	\int_{
		B_R
	}
	u (x, t_2)
	e^{\lambda t_2}
	\,
	dx	dt
	+
	\int_{
		B_{2 R}
	}
	u (x, t_1)
	e^{\lambda t_1}
	\,
	dx	dt
	\\
	&
	\quad
	{}
	+
	\lambda
	\int_{
		B_{R_2} \times (t_1, t_2) \setminus E
	}
	u
	\varphi_0
	\left(
		\frac{R_2 - |x|}{R_2 - R_1}
	\right)
	e^{\lambda t}
	\,
	dx	dt
	\\
	&
	\qquad
	{}
	\ge
	\int_{
		B_{R_2} \times (t_1, t_2) \cap E
	}
	(f (x, t) g (u) - \lambda u)
	\varphi_0
	\left(
		\frac{R_2 - |x|}{R_2 - R_1}
	\right)
	e^{\lambda t}
	\,
	dx dt.
\end{align*}
Combining this with the evident inequalities
$$
	\int_{
		B_{R_2} \times (t_1, t_2) \setminus E
	}
	u
	\varphi_0
	\left(
		\frac{R_2 - |x|}{R_2 - R_1}
	\right)
	e^{\lambda t}
	\,
	dx	dt
	\le
	\int_{
		B_{R_2} \times (t_1, t_2) \setminus E
	}
	u
	e^{\lambda t}
	\,
	dx	dt
$$
and
\begin{align*}
	&
	\int_{
		B_{R_2} \times (t_1, t_2) \cap E
	}
	(f (x, t) g (u) - \lambda u)
	\varphi_0
	\left(
		\frac{R_2 - |x|}{R_2 - R_1}
	\right)
	e^{\lambda t}
	\,
	dx dt
	\\
	&
	\quad
	{}
	\ge
	\frac{1}{2}
	\int_{
		B_{R_2} \times (t_1, t_2) \cap E
	}
	f (x, t) g (u)
	\varphi_0
	\left(
		\frac{R_2 - |x|}{R_2 - R_1}
	\right)
	e^{\lambda t}
	\,
	dx dt
	\ge
	\frac{1}{2}
	\int_{
		B_{R_1} \times (t_1, t_2) \cap E
	}
	f (x, t) g (u)
	e^{\lambda t}
	\,
	dx dt,
\end{align*}
we complete the proof.
\end{proof}

\begin{Lemma}\label{L3.2}
Suppose that $\overline{B_{2 R}^y} \subset \Omega$ for some $y \in \Omega$ and $R > 0$.
Also let $\varepsilon > 0$, $\lambda > 0$, and $t_1 \in T (B_{2 R}^y)$ be real numbers 
such that
\begin{equation}
	f (x, t) g (\zeta) > 2 \lambda \zeta
	\label{L3.2.1}
\end{equation}
for almost all $(x, t) \in B_{2 R}^y \times (t_1, \infty)$ 
and for all $\zeta \in [\varepsilon, \infty)$.
Then
\begin{equation}
	\int_{B_R^y}
	u (x, t_2)
	\,
	dx
	\le
	e^{ - \lambda (t_2 - t_1)}
	\int_{
		B_{2 R}^y
	}
	u (x, t_1)
	\,
	dx
	+
	C
	\left(
		\varepsilon
		R^n
		+
		\frac{
			\varepsilon^p
			R^{n - m}
		}{
			\lambda
		}
		+
		\frac{
			R^{
				n - m / (1 - p)
			}
		}{
			\lambda^{1 / (1 - p)}
		}
	\right)
	\label{L3.1.3}
\end{equation}
for almost all real numbers $t_2 > t_1$, 
where the constant $C > 0$ depends only on $A$, $m$, $n$, and~$p$.
\end{Lemma}

\begin{proof}
Let us denote by $C$ various positive constants that can depend only on $A$, $m$, $n$, and~$p$.
Without loss of generality, we can limit ourselves to the case of $y = 0$. In this case, estimate~\eqref{L3.1.3} takes the form
\begin{equation}
	\int_{B_R}
	u (x, t_2)
	\,
	dx
	\le
	e^{ - \lambda (t_2 - t_1)}
	\int_{
		B_{2 R}
	}
	u (x, t_1)
	\,
	dx
	+
	C
	\left(
		\varepsilon
		R^n
		+
		\frac{
			\varepsilon^p
			R^{n - m}
		}{
			\lambda
		}
		+
		\frac{
			R^{
				n - m / (1 - p)
			}
		}{
			\lambda^{1 / (1 - p)}
		}
	\right)
	\label{PT2.1.2}
\end{equation}
for almost all $t_2 > t_1$.
It can also be assumed that 
\begin{equation}
	\| u \|_{
		L_1 (B_{3 R / 2} \times (t_1, t_2))
	}
	>
	0;
	\label{PT2.1.1}
\end{equation}
for all $t_2 > t_1$; otherwise we replace $t_1$ by 
$$
	t_* 
	= 
	\sup 
	\{ 
		t > t_1 
		: 
		\| u \|_{
			L_1 (B_{3 R / 2} \times (t_1, t))
		}
		=
		0
	\}.
$$
It is clear that the left-hand side of~\eqref{PT2.1.2} is equal to zero
for almost all $t_1 < t_2 < t_*$.
In so doing, if $t_* = \infty$, then~\eqref{PT2.1.2} is obvious 
as the left-hand side of this inequality is equal to zero for almost all $t_2 > t_1$.

Let $t_2 \in T(B_R)$ be a real number with $t_2 > t_1$.
We denote
$$
	J (r)
	=
	\int_{t_1}^{t_2}
	\int_{
		B_r
	}
	e^{\lambda t}
	u^p
	\,
	dx dt,
	\quad
	0 < r \le 2 R.
$$
From~\eqref{PT2.1.1}, it follows that $J (r) > 0$ for all $r \ge 3 R / 2$.
Let us construct a finite sequence of real numbers $r_1 < r_2 < \ldots < r_l$ by induction.
Take $r_1 = 3 R / 2$.
Now, assume that $r_i$ is already known. 
In the case of $r_i \ge 7 R / 4$, we put $l = i$ and stop; otherwise we take
$$
	r_{i + 1} 
	= 
	\sup
	\{
		r \in (r_i, 2 r_i) \cap [3 R / 2, 2 R]
		:
		J (r) \le 2 J (r_i)
	\}.
$$
Further, let
$$
	E 
	= 
	\{ 
		(x, t) \in B_{2 R} \times (t_1, t_2)
		:
		u (x, t) \ge \varepsilon
	\}.
$$
As condition~\eqref{L3.2.1} implies~\eqref{L3.1.2}, using Lemma~\ref{L3.1}, we have
\begin{align}
	&
	\int_{
		B_{2 R}
	}
	e^{\lambda t_1}
	u (x, t_1)
	\,
	dx
	-
	\int_{B_R}
	e^{\lambda t_2}
	u (x, t_2)
	\,
	dx
	+
	\lambda
	\int_{
		B_{r_{i + 1}} \times (t_1, t_2)
		\setminus
		E
	}
	e^{\lambda t}
	u
	\,
	dx dt
	\nonumber
	\\
	&
	\qquad
	{}
	+
	\frac{
		C
	}{
		(r_{i + 1} - r_i)^m
	}
	\int_{t_1}^{t_2}
	\int_{
		B_{r_{i + 1}} \setminus B_{r_i}
	}
	e^{\lambda t}
	u^p
	\,
	dx	dt
	\ge
	\frac{1}{2}
	\int_{
		B_{r_i} \times (t_1, t_2)
		\cap
		E
	}
	e^{\lambda t}
	f (x, t) g (u)
	\,
	dx dt
	\label{PT2.1.3}
\end{align}
for all $i = 1, 2, \ldots, l - 1$.
If
$$
	\int_{
		B_{2 R}
	}
	e^{\lambda t_1}
	u (x, t_1)
	\,
	dx
	-
	\int_{B_R}
	e^{\lambda t_2}
	u (x, t_2)
	\,
	dx
	+
	\lambda 
	\int_{
		B_{r_{i + 1}} \times (t_1, t_2)
		\setminus
		E
	}
	e^{\lambda t}
	u
	\,
	dx dt
	\ge
	0
$$
for some $1 \le i \le l - 1$, then 
$$
	\int_{B_R}
	e^{\lambda t_2}
	u (x, t_2)
	\,
	dx
	\le
	\int_{
		B_{2 R}
	}
	e^{\lambda t_1}
	u (x, t_1)
	\,
	dx
	+
	\lambda 
	\int_{
		B_{r_{i + 1}} \times (t_1, t_2)
		\setminus
		E
	}
	e^{\lambda t}
	u
	\,
	dx dt
	\le
	\int_{
		B_{2 R}
	}
	e^{\lambda t_1}
	u (x, t_1)
	\,
	dx
	+
	C
	\varepsilon
	e^{\lambda t_2}
	R^n,
$$
whence~\eqref{PT2.1.2} follows at once; therefore, it can be assumed that
$$
	\int_{
		B_{2 R}
	}
	e^{\lambda t_1}
	u (x, t_1)
	\,
	dx
	-
	\int_{B_R}
	e^{\lambda t_2}
	u (x, t_2)
	\,
	dx
	+
	\lambda 
	\int_{
		B_{r_{i + 1}} \times (t_1, t_2)
		\setminus
		E
	}
	e^{\lambda t}
	u
	\,
	dx dt
	<
	0
$$
for all $i = 1, 2, \ldots, l - 1$. By~\eqref{PT2.1.3}, this implies the inequality
\begin{equation}
	\int_{t_1}^{t_2}
	\int_{
		B_{r_{i + 1}} \setminus B_{r_i}
	}
	e^{\lambda t}
	u^p
	\,
	dx	dt
	\ge
	C
	(r_{i + 1} - r_i)^m
	\int_{
		B_{r_i} \times (t_1, t_2)
		\cap
		E
	}
	e^{\lambda t}
	f (x, t) g (u)
	\,
	dx dt
	\label{PT2.1.4}
\end{equation}
for all $i = 1, 2, \ldots, l - 1$.
From~~\eqref{L3.2.1}, it follows that
$$
	\int_{
		B_{r_i} \times (t_1, t_2)
		\cap
		E
	}
	e^{\lambda t}
	f (x, t) g (u)
	\,
	dx dt
	\ge
	2 
	\lambda
	\int_{
		B_{r_i} \times (t_1, t_2)
		\cap
		E
	}
	e^{\lambda t}
	u
	\,
	dx dt.
$$
Combining this with~\eqref{PT2.1.4}, we obtain
\begin{equation}
	\int_{t_1}^{t_2}
	\int_{
		B_{r_{i + 1}} \setminus B_{r_i}
	}
	e^{\lambda t}
	u^p
	\,
	dx	dt
	\ge
	C 
	\lambda
	(r_{i + 1} - r_i)^m
	\int_{
		B_{r_i} \times (t_1, t_2)
		\cap
		E
	}
	e^{\lambda t}
	u
	\,
	dx dt
	\label{PT2.1.5}
\end{equation}
for all $i = 1, 2, \ldots, l - 1$.
By the H\"older inequality,
\begin{align*}
	\int_{
		B_{r_i} \times (t_1, t_2)
		\cap
		E
	}
	e^{\lambda t}
	u^p
	\,
	dx dt
	&
	\le
	\left(
		\int_{
			B_{r_i} \times (t_1, t_2)
			\cap
			E
		}
		e^{\lambda t}
		\,
		dx dt
	\right)^{1 - p}
	\left(
		\int_{
			B_{r_i} \times (t_1, t_2)
			\cap
			E
		}
		e^{\lambda t}
		u
		\,
		dx dt
	\right)^p
	\\
	&
	\le
	\frac{
		C
		e^{(1 - p) \lambda t_2}
		R^{(1 - p) n}
	}{
		\lambda^{1 - p}
	}
	\left(
		\int_{
			B_{r_i} \times (t_1, t_2)
			\cap
			E
		}
		e^{\lambda t}
		u
		\,
		dx dt
	\right)^p;
\end{align*}
therefore,~\eqref{PT2.1.5} allows us to assert that
\begin{equation}
	\int_{t_1}^{t_2}
	\int_{
		B_{r_{i + 1}} \setminus B_{r_i}
	}
	e^{\lambda t}
	u^p
	\,
	dx	dt
	\ge
	\frac{
		C 
		\lambda^{1 / p}
		(r_{i + 1} - r_i)^m
	}{
		\left(
			e^{\lambda t_2} 
			R^n
		\right)^{(1 - p) / p}
	}
	\left(
		\int_{
			B_{r_i} \times (t_1, t_2)
			\cap
			E
		}
		e^{\lambda t}
		u^p
		\,
		dx dt
	\right)^{1 / p}
	\label{PT2.1.6}
\end{equation}
for all $i = 1, 2, \ldots, l - 1$.
At first, let
\begin{equation}
	\int_{
		B_{r_i} \times (t_1, t_2)
		\setminus
		E
	}
	e^{\lambda t}
	u^p
	\,
	dx dt
	\ge
	\frac{1}{2}
	\int_{t_1}^{t_2}
	\int_{
		B_{r_i}
	}
	e^{\lambda t}
	u^p
	\,
	dx	dt
	\label{PT2.1.13}
\end{equation}
for some $i = 1, 2, \ldots, l - 1$. 
Applying Lemma~\ref{L3.1} with $R_1 = R$ and $R_2 = 3 R / 2$, we arrive at the estimate
\begin{align}
	\int_{B_R}
	e^{\lambda t_2}
	u (x, t_2)
	\,
	dx
	\le
	{}
	&
	\int_{
		B_{2 R}
	}
	e^{\lambda t_1}
	u (x, t_1)
	\,
	dx
	+
	\lambda 
	\int_{
		B_{3 R / 2} \times (t_1, t_2)
		\setminus
		E
	}
	e^{\lambda t}
	u
	\,
	dx dt
	\nonumber
	\\
	&
	{}
	+
	\frac{
		C
	}{
		R^m
	}
	\int_{
		(B_{3 R / 2}	\setminus B_R) \times (t_1, t_2)
	}
	e^{\lambda t}
	u^p
	\,
	dx	dt.
	\label{PT2.1.8}
\end{align}
Since
$$
	\int_{
		(B_{3 R / 2}	\setminus B_R) \times (t_1, t_2)
	}
	e^{\lambda t}
	u^p
	\,
	dx	dt
	\le
	\int_{t_1}^{t_2}
	\int_{
		B_{r_i}
	}
	e^{\lambda t}
	u^p
	\,
	dx	dt
	\le
	2
	\int_{
		B_{r_i} \times (t_1, t_2)
		\setminus
		E
	}
	e^{\lambda t}
	u^p
	\,
	dx dt,
$$
this yields
\begin{align*}
	&
	\int_{B_R}
	e^{\lambda t_2}
	u (x, t_2)
	\,
	dx
	\le
	\int_{
		B_{2 R}
	}
	e^{\lambda t_1}
	u (x, t_1)
	\,
	dx
	+
	\lambda 
	\int_{
		B_{3 R / 2} \times (t_1, t_2)
		\setminus
		E
	}
	e^{\lambda t}
	u
	\,
	dx dt
	+
	\frac{
		C
	}{
		R^m
	}
	\int_{
		B_{r_i} \times (t_1, t_2)
		\setminus
		E
	}
	e^{\lambda t}
	u^p
	\,
	dx dt
	\\
	&
	\quad
	{}
	\le
	\int_{
		B_{2 R}
	}
	e^{\lambda t_1}
	u (x, t_1)
	\,
	dx
	+
	C
	e^{\lambda t_2}
	\left(
		\varepsilon
		R^n
		+
		\frac{
			\varepsilon^p
			R^{n - m}
		}{
			\lambda
		}
	\right),
\end{align*}
whence~\eqref{PT2.1.2} follows.
Now, assume that for any $i = 1, 2, \ldots, l - 1$ relation~\eqref{PT2.1.13} is not satisfied.
In this case, we obviously have
$$
	\int_{
		B_{r_i} \times (t_1, t_2)
		\cap
		E
	}
	e^{\lambda t}
	u^p
	\,
	dx dt
	\ge
	\frac{1}{2}
	\int_{t_1}^{t_2}
	\int_{
		B_{r_i}
	}
	e^{\lambda t}
	u^p
	\,
	dx	dt
$$
for all $i = 1, 2, \ldots, l - 1$. 
Combining the last expression with~\eqref{PT2.1.6}, one can conclude that
$$
	\int_{t_1}^{t_2}
	\int_{
		B_{r_{i + 1}} \setminus B_{r_i}
	}
	e^{\lambda t}
	u^p
	\,
	dx	dt
	\ge
	\frac{
		C 
		\lambda^{1 / p}
		(r_{i + 1} - r_i)^m
	}{
		\left(
			e^{\lambda t_2} 
			R^n
		\right)^{(1 - p) / p}
	}
	\left(
		\int_{t_1}^{t_2}
		\int_{
			B_{r_i}
		}
		e^{\lambda t}
		u^p
		\,
		dx	dt
	\right)^{1 / p}
$$
or, in other words,
\begin{equation}
	J (r_{i+1}) - J (r_i)
	\ge
	\frac{
		C 
		\lambda^{1 / p}
		(r_{i + 1} - r_i)^m
	}{
		\left(
			e^{\lambda t_2} 
			R^n
		\right)^{(1 - p) / p}
	}
	J^{1 / p} (r_i)
	\label{PT2.1.9}
\end{equation}
for all $i = 1, 2, \ldots, l - 1$.

If $r_l = 2 R$, then $r_l - r_{l - 1} \ge R / 4$.
Hence,~\eqref{PT2.1.9} implies the inequality
$$
	\frac{
		J (r_l) - J (r_{l - 1})
	}{
		J^{1 / p} (r_{l - 1})
	}
	\ge
	\frac{
		C 
		\lambda^{1 / p}
		R^{
			m - n (1 - p) / p
		}
	}{
		e^{\lambda t_2 (1 - p) / p}
	}.
$$
As $2 J (r_{l - 1}) \ge J (r_l)$, this yields
$$
	J^{1 - 1 / p} (r_{l - 1})
	\ge
	\frac{
		C 
		\lambda^{1 / p}
		R^{
			m - n (1 - p) / p
		}
	}{
		e^{\lambda t_2 (1 - p) / p}
	},
$$
whence it follows that
$$
	\int_{
		(B_{3 R / 2} \setminus B_R) \times (t_1, t_2)
	}
	e^{\lambda t}
	u^p
	\,
	dx dt
	\le
	J (r_{l-1})
	\le
	\frac{
		C
		e^{\lambda t_2} 
		R^{
			n - m p / (1 - p)
		}
	}{
		\lambda^{1 / (1 - p)}
	};
$$
therefore, taking into account~\eqref{PT2.1.8}, we arrive at the estimate
\begin{equation}
	\int_{B_R}
	e^{\lambda t_2}
	u (x, t_2)
	\,
	dx
	\le
	\int_{
		B_{2 R}
	}
	e^{\lambda t_1}
	u (x, t_1)
	\,
	dx
	+
	C
	e^{\lambda t_2}
	\left(
		\varepsilon
		R^n
		+
		\frac{
			R^{
				n - m / (1 - p)
			}
		}{
			\lambda^{1 / (1 - p)}
		}
	\right)
	\label{PT2.1.10}
\end{equation}
which in turn leads us to~\eqref{PT2.1.2}.

Now, let $r_l < 2 R$. In this case, we have $J (r_{i+1}) = 2 J (r_i)$ for all $i = 1, 2, \ldots, l - 1$.
Consequently,~\eqref{PT2.1.9} implies that
$$
	\frac
	{
		1
	}{
		J^{
			(1 - p) / (p m)
		}
		(r_i)
	}
	=
	\left(
		\frac{
			J (r_{i+1}) - J (r_i)
		}{
			J^{1 / p} (r_i)
		}
	\right)^{1 / m}
	\ge
	\frac{
		C 
		\lambda^{1 / (p m)}
		(r_{i + 1} - r_i)
	}{
		\left(
			e^{\lambda t_2} 
			R^n
		\right)^{(1 - p) / (p m)}
	},
	\quad
	i = 1, 2, \ldots, l - 1,
$$
whence we obtain
$$
	\sum_{i=1}^{l-1}
	\frac
	{
		1
	}{
		J^{
			(1 - p) / (p m)
		}
		(r_i)
	}
	\ge
	\frac{
		C 
		\lambda^{1 / (p m)}
		(r_l - r_1)
	}{
		\left(
			e^{\lambda t_2} 
			R^n
		\right)^{(1 - p) / (p m)}
	}.
$$
Since $r_l - r_1 \ge R / 4$ and, moreover, $J (r_i) = 2^{i - 1} J (3 R / 2)$ for all $i = 1, 2, \ldots, l - 1$, 
this yields
$$
	J (3 R / 2)
	\le
	\frac{
		C
		e^{\lambda t_2} 
		R^{
			n - m p / (1 - p)
		}
	}{
		\lambda^{1 / (1 - p)}
	}.
$$
Combining the last estimate with~\eqref{PT2.1.8}, we again arrive at~\eqref{PT2.1.10} 
and therefore at~\eqref{PT2.1.2}.
\end{proof}

\begin{proof}[Proof of Theorem~\ref{T2.1}]
As previously, by $C$ we mean various positive constants which can 
depend only on $A$, $m$, $n$, and~$p$.
Consider a finite cover of the compact set $K$ by open balls $B_{R_i}^{y_i}$, $i = 1,2,\ldots,N$,
such that
$$
	\bigcup_{i = 1}^N
	\overline{B_{2 R_i}^{y_i}}
	\subset 
	\Omega.
$$
Let $\varepsilon > 0$ and $\lambda > 0$ be some real numbers.
Since $g$ is a positive continuous function on the interval $[\varepsilon, \infty)$ 
satisfying condition~\eqref{T2.1.1}, we have
\begin{equation}
	g (\zeta) \ge \delta \zeta
	\label{PT2.1.11}
\end{equation}
with some constant $\delta > 0$ for all $\zeta \in [\varepsilon, \infty)$.
In view of~\eqref{T2.1.2}, there exists $t_1 \in \bigcap_{i=1}^N T (B_{2 R_i}^{y_i})$ such that
\begin{equation}
	\delta
	\mathop{\rm ess\,inf}_{
		\Omega \times (t_1, \infty)
	}
	f 
	>
	2
	\lambda.
	\label{PT2.1.12}
\end{equation}
It is easy to see
that conditions~\eqref{PT2.1.11} and~\eqref{PT2.1.12}
guarantee the validity of inequality~\eqref{L3.2.1}
for almost all $(x, t) \in \Omega \times (t_1, \infty)$ 
and for all $\zeta \in [\varepsilon, \infty)$.
Thus, taking into account Lemma~\ref{L3.2}, we obtain
$$
	\int_{B_{R_i}^{y_i}}
	u (x, t_2)
	\,
	dx
	\le
	e^{ - \lambda (t_2 - t_1)}
	\int_{
		B_{2 R_i}^{y_i}
	}
	u (x, t_1)
	\,
	dx
	+
	C
	\left(
		\varepsilon
		R_i^n
		+
		\frac{
			\varepsilon^p
			R_i^{n - m}
		}{
			\lambda
		}
		+
		\frac{
			R_i^{
				n - m / (1 - p)
			}
		}{
			\lambda^{1 / (1 - p)}
		}
	\right),
	\quad
	i = 1,2,\ldots,N,
$$
for almost all real numbers $t_2 > t_1$, whence it follows that
$$
	\int_K
	u (x, t_2)
	\,
	dx
	\le
	e^{ - \lambda (t_2 - t_1)}
	\sum_{i=1}^N
	\int_{
		B_{2 R_i}^{y_i}
	}
	u (x, t_1)
	\,
	dx
	+
	C
	\sum_{i=1}^N
	\left(
		\varepsilon
		R_i^n
		+
		\frac{
			\varepsilon^p
			R_i^{n - m}
		}{
			\lambda
		}
		+
		\frac{
			R_i^{
				n - m / (1 - p)
			}
		}{
			\lambda^{1 / (1 - p)}
		}
	\right)
$$
for almost all real numbers $t_2 > t_1$.
Passing to the limit in the last estimate, we have
\begin{equation}
	\limsup_{t_2 \to \infty}
	\int_K
	u (x, t_2)
	\,
	dx
	\le
	C
	\sum_{i=1}^N
	\left(
		\varepsilon
		R_i^n
		+
		\frac{
			\varepsilon^p
			R_i^{n - m}
		}{
			\lambda
		}
		+
		\frac{
			R_i^{
				n - m / (1 - p)
			}
		}{
			\lambda^{1 / (1 - p)}
		}
	\right).
	\label{PT2.1.14}
\end{equation}
As we can take $\varepsilon > 0$ arbitrarily small and $\lambda > 0$ arbitrarily large,
this completes the proof.
\end{proof}

\begin{proof}[Proof of Theorem~\ref{T2.2}]
Let $B_{R_i}^{y_i}$, $i = 1,2,\ldots,N$, be a finite cover of the compact set $K$
given in the proof of Theorem~\ref{T2.1}.
Since $g$ is a positive continuous function on the interval $[1, \infty)$ 
satisfying condition~\eqref{T2.1.1}, there exists a real number $\delta > 0$ 
such that inequality~\eqref{PT2.1.11} is valid for all $\zeta \in [1, \infty)$.
Take $t_1 \in \bigcap_{i=1}^N T (B_{2 R_i}^{y_i})$ satisfying the condition
$$
	\mathop{\rm ess\,inf}_{
		B_\Omega \times (t_1, \infty)
	}
	f 
	\ge
	\frac{\gamma}{2}.
$$
The last inequality immediately implies~\eqref{PT2.1.12}, where
$$
	\lambda = \frac{\delta \gamma}{4}.
$$
Thus, repeating the arguments given in the proof of~\eqref{PT2.1.14} with $\varepsilon = 1$, 
we obtain
$$
	\limsup_{t_2 \to \infty}
	\int_K
	u (x, t_2)
	\,
	dx
	\le
	C
	\sum_{i=1}^N
	\left(
		R_i^n
		+
		\frac{
			R_i^{n - m}
		}{
			\lambda
		}
		+
		\frac{
			R_i^{
				n - m / (1 - p)
			}
		}{
			\lambda^{1 / (1 - p)}
		}
	\right),
$$
where the constant $C > 0$ depends only on $A$, $m$, $n$, and~$p$.
\end{proof}

\begin{center} \bf Acknowledgments \end{center}

The work is supported by RUDN University, Project 5-100. 
The work of the first author is also supported by RSF, grant 20-11-20272.

\end{document}